\documentclass[lettersize,journal]{IEEEtran}
\usepackage{xcolor}
\usepackage[hidelinks]{hyperref}
\usepackage{algorithm}
\usepackage{algpseudocode}
\usepackage{subcaption}
\usepackage{multirow}
\usepackage{amsthm}
\usepackage{makecell}
\usepackage[switch]{lineno}

\usepackage[T1]{fontenc}%

\ifCLASSINFOpdf
\else
\fi

\usepackage{amsmath}
\usepackage{amsfonts}

\usepackage{bm}

\usepackage[some]{background}
\SetBgScale{1}
\SetBgContents{\parbox{20cm}{%
  \Huge Draft:  \today\\[18cm]\rotatebox{180}{\Huge Draft:  \today}}}
\SetBgColor{red}
\SetBgAngle{270}
\SetBgOpacity{0.2}

\newtheorem{theorem}{Theorem}
\newtheorem{remark}{Remark}
\newtheorem{corollary}{Corollary}

\begin{document}

\title{Circular Microalgae-Based Carbon Control for Net Zero}

\author{Federico~Zocco, Joan Garc\'ia, and Wassim M. Haddad
\thanks{F.~Zocco is with the Research Centre in Sustainable Energy, School of Mechanical and Aerospace Engineering, Queen's University Belfast, Northern Ireland, and with the Centre for Sustainable Manufacturing and Recycling Technologies (SMART), School of Mechanical, Electrical and Manufacturing Engineering, Loughborough University, England, UK. Email: federico.zocco.fz@gmail.com.}
\thanks{\emph{(Corresponding authors: Federico Zocco)}}%
\thanks{J.~Garc\'ia is with GEMMA - Group of Environmental Engineering and Microbiology, Department of Civil and Environmental Engineering, Universitat Polit\`ecnica de Catalunya - BarcelonaTech,
c/ Jordi Girona 1-3, Building D1, E-08034 Barcelona, Spain. Email: joan.garcia@upc.edu.}
\thanks{W. M. Haddad is with the School of Aerospace Engineering, Georgia Institute of Technology, Atlanta, Georgia, USA. Email: wm.haddad@aerospace.gatech.edu.}}

\maketitle

\begin{abstract}
The alteration of the climate in various areas of the world is of increasing concern since climate stability is a necessary condition for human survival as well as every living organism. The main reason of climate change is the greenhouse effect caused by the accumulation of carbon dioxide in the atmosphere. In this paper, we design a networked system underpinned by compartmental dynamical thermodynamics to circulate the atmospheric carbon dioxide. Specifically, in the carbon dioxide emitter compartment, we develop an initial-condition-dependent finite-time stabilizing controller that guarantees stability within a desired time leveraging the system property of affinity in the control. Then, to compensate for carbon emissions we show that a cultivation of microalgae with a volume 625 times bigger than the one of the carbon emitter is required. To increase the carbon uptake of the microalgae, we implement the nonaffine-in-the-control microalgae dynamical equations as an environment of a state-of-the-art library for reinforcement learning (RL), namely, Stable-Baselines3, and then, through the library, we test the performance of eight RL algorithms for training a controller that maximizes the microalgae absorption of carbon through the light intensity. All the eight controllers increased the carbon absorption of the cultivation during a training of 200,000 time steps with a maximum episode length of 200 time steps and with no termination conditions. This work is a first step towards approaching net zero as a classical and learning-based network control problem. The source code is publicly available\footnotemark{}\footnotetext{\url{https://github.com/fedezocco/TMN5compCO2-SciPy}}.
\end{abstract}


\IEEEpeerreviewmaketitle

\section{Introduction}

The European Environment Agency reports that, between 1980 and 2020, weather- and climate-related events amounted to between 85,000 and 145,000 fatalities \cite{EEA}. These numbers are small if compared, for example, with fatalities caused by the influenza virus, which are estimated to be at least 40,000 each year in the European Union \cite{ECDCreport}. However, since the annual average anomaly of temperature keeps increasing from 0.85 $^\circ$C in 2018 to 0.89 $^\circ$C in 2022 as monitored by NASA \cite{NASAanomaly}, the occurrences of extreme weather phenomena are expected to rise in the coming years as reported recently: these include extreme precipitations in Northeast United States \cite{jong2023increases}, rainfalls in UK \cite{kendon2023variability}, heat waves worldwide \cite{perkins2020increasing}, and heat waves in the Mediterranean basin \cite{molina2020future}. Hence, strategies to control the atmospheric carbon dioxide are increasingly needed to prevent a climate event horizon, after which limiting a given climate impact becomes geophysically infeasible \cite{aengenheyster2018point}. The term ``net zero'' is typically used by governments and academia to indicate the desired situation in which the carbon dioxide stops to accumulate in the atmosphere \cite{NZdefinition,EU-netZero}.    

The circular economy paradigm is gaining interest as a solution to simultaneously reduce material supply uncertainties and pollution \cite{UN-CE,EU-CE,US-CE}. The practices of a circular economy are, for example, reduce, reuse, repair, and recycle \cite{potting2017circular}. These practices can be applied to any material, with a priority given to those considered as critical \cite{GAUSTAD201824,CRMs-EU}. In particular, the application of the ``reuse'' and ``recycle'' practices to carbon dioxide can reduce its accumulation in the atmosphere by \emph{recirculating} it, and hence, facilitate the achievement of the net zero target mentioned above.         
        
In this context, this paper makes the following contributions: 
\begin{itemize}
\item{We design a network leveraging compartmental dynamical thermodynamics \cite{haddad2019dynamical} to \emph{circulate} the atmospheric carbon dioxide for reaching net zero (covered in Sections \ref{sec:netDesign}, \ref{sec:numStudy}, and \ref{sec:RLcontrol}). While we consider the particular case of an anaerobic digester as the carbon source, our network design approach is generalizable and scalable; it is general because it is based on thermodynamics \cite{haddad2017thermodynamics} and it is scalable because it is based on the discretization of the problem into thermodynamic compartments that can be added, removed, and modified as appropriate \cite{zocco2023thermodynamical,zocco2024unification}.} 
\item{We revised the finite-time stabilizing controller proposed in \cite{haddad2015finite} to achieve an \emph{initial-condition-dependent} formulation valid for a family of Lyapunov candidate functions and implemented it to regulate the carbon emitter (Section \ref{sec:netDesign}).}
\item{We designed and compared eight reinforcement-learning (RL) controllers trained to maximize the uptake of carbon dioxide of the microalgae through the light intensity (Section \ref{sec:RLcontrol}). To the best of our knowledge, this is the first work that approaches net zero as a classical and learning-based network control problem.}
\end{itemize}

Throughout the paper, matrices and vectors are indicated with bold capital and lower case letters, respectively, while sets are indicated with calligraphic letters. 

The rest of the paper is organized as follows. Section \ref{sec:relWork} discusses the related work, Section \ref{sec:netDesign} covers the design of the network, Section \ref{sec:numStudy} considers a numerical study, Section \ref{sec:RLcontrol} reports the design of the RL controllers for the microalgae system, and finally Section \ref{sec:concl} gives the conclusions. %

\section{Related Work}\label{sec:relWork}
\subsection{Plants for Greenhouse Effect Mitigation}
To limit the carbon accumulation in the atmosphere, three main methods are possible, namely, reducing the use of combustion for energy production, removing $\text{CO}_2$ from the atmosphere, and capturing $\text{CO}_2$ at the source before it enters the atmosphere \cite{benemann1997co2}. Traditional strategies focus on mitigating the emissions at the source \cite{shafique2020overview}; Shafique \textit{et al.} \cite{shafique2020overview} provide an overview of carbon sequestration using plants installed on urban roofs. In contrast, our paper considers the adversarial actions of a source and a sink to regulate the resulting carbon concentration in the atmosphere. Zhang \textit{et al.} \cite{zhang2022urban} assessed whether urban green spaces in China were sinks or sources of carbon using a life-cycle assessment (LCA) approach. They found that trees and shrubs were sinks, whereas lawns were sources due to required maintenance. Marchi \textit{et al.} \cite{marchi2015carbon} developed a model of a vertical greenery system and they compared the sequestration efficiency of different types of plants. Wang \textit{et al.} \cite{wang2021promoting} evaluated urban planting designs considering the sequestration efficiency of plants. Medium-sized evergreen trees were among the best performing species.    
    
The direct injection of $\text{CO}_2$ into the deep oceans, aquifers or depleted oil wells is a method for large-scale carbon sequestration, but it requires particular geological conditions \cite{zhou2017bio}. In contrast, forestation is a more natural process than artificial injection, but it has limited sequestration capacity unless large land areas are covered \cite{zhou2017bio}. Chemical absorption through neutralization of carbonic acid to form carbonates or bicarbonates provides a safe and permanent sequestration, but it is energy and cost intensive \cite{zhou2017bio}. In this paper, we focus on sequestration using microalgae, which is more efficient than forestation. 

Currently, the disadvantages of microalgae are the relatively high costs and their sensitivity to toxic substances in exhaust gases \cite{zhou2017bio}. To address the issue of cost, Ramaraj \textit{et al.} \cite{ramaraj2015biomass} studied the growth of algae in natural water, which is cheaper than their cultivation in an artificial medium; the volumetric carbon uptake rate was $175 \pm 27.86 \text{ mg}\text{L}^{-1}\text{d}^{-1}$. Viswanaathan \textit{et al.} \cite{viswanaathan2022integrated} tabulated the $\text{CO}_2$ uptake of different species of microalgae.  

Another benefit of microalgae-based sequestration is their use as biomass to produce biofuels \cite{arun2021technical}. Hence, plant-based sequestration methods may facilitate a circular flow of material compared to non-biological methods, i.e., they have a life-cycle that easily aligns with the circularity principles detailed in \cite{suarez2019operational}. It is estimated that the major costs in biofuel production from algal cultivations are 77\% from culturing, 12\% from harvesting, and 7.9\% from lipid extraction \cite{sarwer2022algal}.

\subsection{Thermodynamical Material Networks for a Circular Economy}
At the core of any engineering subject there is the application of principles of physics and chemistry to develop a mathematical description of the target system. One of the most delicate steps is defining the conditions and simplifications that make the mathematical description sufficiently accurate for the desired purpose, but also simple enough to conclude with quantitative results and meaningful physical interpretations. This engineering approach is systematically used to design machines and processes, whereas it is rarely used for tackling whole-system problems such as the design of circular supply-recovery chains in a circular economy \cite{EMAFund}. Indeed, to date, circular-economy-related models are usually developed with data-analysis techniques such as LCA \cite{amicarelli2022life,walker2020life,xia2022review,cucurachi2019life,sala2021evolution} and material flow analysis (MFA) \cite{luan2021dynamic,li2022uncovering,sieber2020dynamic,liu2021dynamic,eriksen2020dynamic}. 

Thermodynamical material networks (TMNs) were recently proposed \cite{zocco2023thermodynamical,zocco2024unification,zocco2022circularity,zocco2024circular} to develop circular-economy models using the traditional engineering approach described above. Specifically, the methodology of TMNs is a generalization of the design of the Rankine cycle, in which mass and energy balances are applied to each compartment of the cycle and where the compartments are connected to form a network that delivers the desired flow of material in space and time (the material is the working fluid in the Rankine cycle). In contrast with data-analysis techniques such as LCA and MFA, TMNs are based on ordinary differential equations derived from dynamical mass and energy balances applied to each thermodynamic compartment of the network, thus increasing the accuracy of the models while being less data intensive \cite{zocco2023thermodynamical,zocco2022circularity}. Moreover, since based on differential equations, the design of TMNs can include one or more control systems. In this paper, we illustrate the use of the TMN methodology for circular microalgae-based carbon control, where the circularity of carbon dioxide is quantified using the TMN-based definition of circularity given in \cite{zocco2024circular} (specifically, in Definition 4), namely, $\lambda(\mathcal{N})$, with $\mathcal{N}$ the TMN processing the target material (carbon dioxide in this case) and with $\lambda(\mathcal{N}) \in (-\infty,0]$. Hence, the circularity of carbon dioxide reduces to the following arg-max problem \cite{zocco2024circular}:
\begin{equation}\label{eq:circProblem}
\mathcal{N}^* = \arg \max \,\,\, \lambda(\mathcal{N}).  
\end{equation}

\subsection{Finite-Time Stabilizing Control}
An essential requirement of a reliable process or system is to work in the desired conditions. If the dynamics of the process or system is described by a set of ordinary differential equations (ODEs), then this essential requirement is met if the system state can be stabilized to the desired conditions. Moreover, in practice, a further requirement is to reach the desired point within a finite time rather than merely asymptotically. The satisfaction of these requirement has led to the development of \emph{finite-time stabilizing} controllers, the first of which was proposed in \cite{Roxin1966} and subsequently extended for time-varying nonlinear dynamical systems in \cite{moulay2008finite,haddad2008finite}, for second-order systems in \cite{bhat1998continuous}, and using output feedback in \cite{hong2001output}.    

If, in addition, the controller solves an optimal control problem, the closed-loop nonlinear system is guaranteed to stabilize to the desired conditions both optimally and within a finite time. Such controllers were proposed in \cite{haddad2015finite} for continuous-time systems, \cite{haddad2023finite} for discrete-time systems, and \cite{lee2023finite} for stochastic discrete-time systems. In this paper, we modify the framework for continuous-time dynamical systems \cite{haddad2015finite} to develop a controller that guarantees optimal finite-time stabilization within a \emph{chosen} time.

\subsection{Reinforcement Learning for Continuous-Time Control}          
The advances in deep learning over the last ten years have led to state-of-the-art reinforcement learning (RL) algorithms based on deep neural networks. In 2015, Schulman \emph{et al.} \cite{pmlr-v37-schulman15} proposed the trust region policy optimization (TRPO)  algorithm for optimizing large nonlinear policies using a convolutional neural network with three layers as the policy and processing raw images directly to solve benchmark games. In 2016, Mnih \emph{et al.} \cite{mnih2016asynchronous} introduced asynchronous gradient descent showing that executing multiple agents in parallel on multiple instances of the environment has a stabilizing effect on training and it can be an alternative to the successful, but memory intensive, experience replay. In the same year, Lillicrap \emph{et al.} \cite{lillicrap2019continuouscontroldeepreinforcement} proposed an actor-critic approach based on the policy gradient algorithm which, with the same network architecture and hyperparameters, solved more than twenty simulated physics tasks. In 2017, Schulman \emph{et al.} \cite{schulman2017proximal} combined the data efficiency and reliability of TRPO while using only first-order optimization and alternate sampling data from the environment with optimizing the objective function via stochastic gradient ascent. The following year, Mania \emph{et al.} \cite{NEURIPS2018_7634ea65} aimed to significantly simplify the overall approach to RL since the complexity of the existing algorithms was, at that time, one of the major barriers to the deployment of RL in controlling real physical systems; they named their algorithm augmented random search (ARS) since the work resulted from an augmentation of a basic random search. Recently, Bhatt \emph{et al.} \cite{bhattcrossq} improved sample efficiency by properly using batch normalization and removing the target networks, Kokolakis \emph{et al.} \cite{kokolakis2023fixed} developed a critic-only RL algorithm for learning the solution to the steady-state Hamilton-Jacobi-Bellman equation in fixed-time, while Abel \emph{et al.} \cite{abel2024definition} provided a careful definition of the emerging concept of \emph{continual} RL as an evolution of the common view of learning as ``finding a solution'' to ``endless adaptation''.

\section{Network Design}\label{sec:netDesign}
The goal of our network is to mitigate the carbon dioxide concentration in the atmosphere using microalgae. To do so, we designed a TMN following the methodology proposed in \cite{zocco2023thermodynamical} (specifically, see Section IV of \cite{zocco2023thermodynamical}), which consists of three main steps. The three steps for this design are detailed as follows. While we consider the particular case of an anaerobic digester as the source of $\text{CO}_2$ emissions, the design approach could be used also with other sources. An anaerobic digester produces biogas which is composed by 60\% methane and 40\% $\text{CO}_2$.

\subsection{Step 1}\label{subsec:Step1} 
As the goal of the network is the local mitigation of carbon dioxide concentration in the neighborhood of an anaerobic digester, we chose $\text{CO}_2$ as the target material. Hence, in this case and in the notation of \cite{zocco2023thermodynamical},
\begin{equation}
\mathcal{B} = \{\beta_1\}, \quad n_\beta = 1,
\end{equation}
where $\beta_1$ is $\text{CO}_2$. 
 
\subsection{Step 2} 
To design the flow of $\mathcal{B}$, we consider the five thermodynamic compartments shown in Fig. \ref{fig:RealisticRepresentation} depicting an anaerobic digester, the atmosphere, a microalgae cultivation located in the neighborhood of the reactor, and two virtual ducts for the exchange of $\text{CO}_2$. One virtual duct is between the reactor and the atmosphere, while the other one is between the microalgae cultivation and the atmosphere. The red arrows indicate the direction of $\text{CO}_2$ flow in each virtual duct.                        
\begin{figure}
\begin{subfigure}{0.5\textwidth}
\centering
\includegraphics[width=0.9\textwidth]{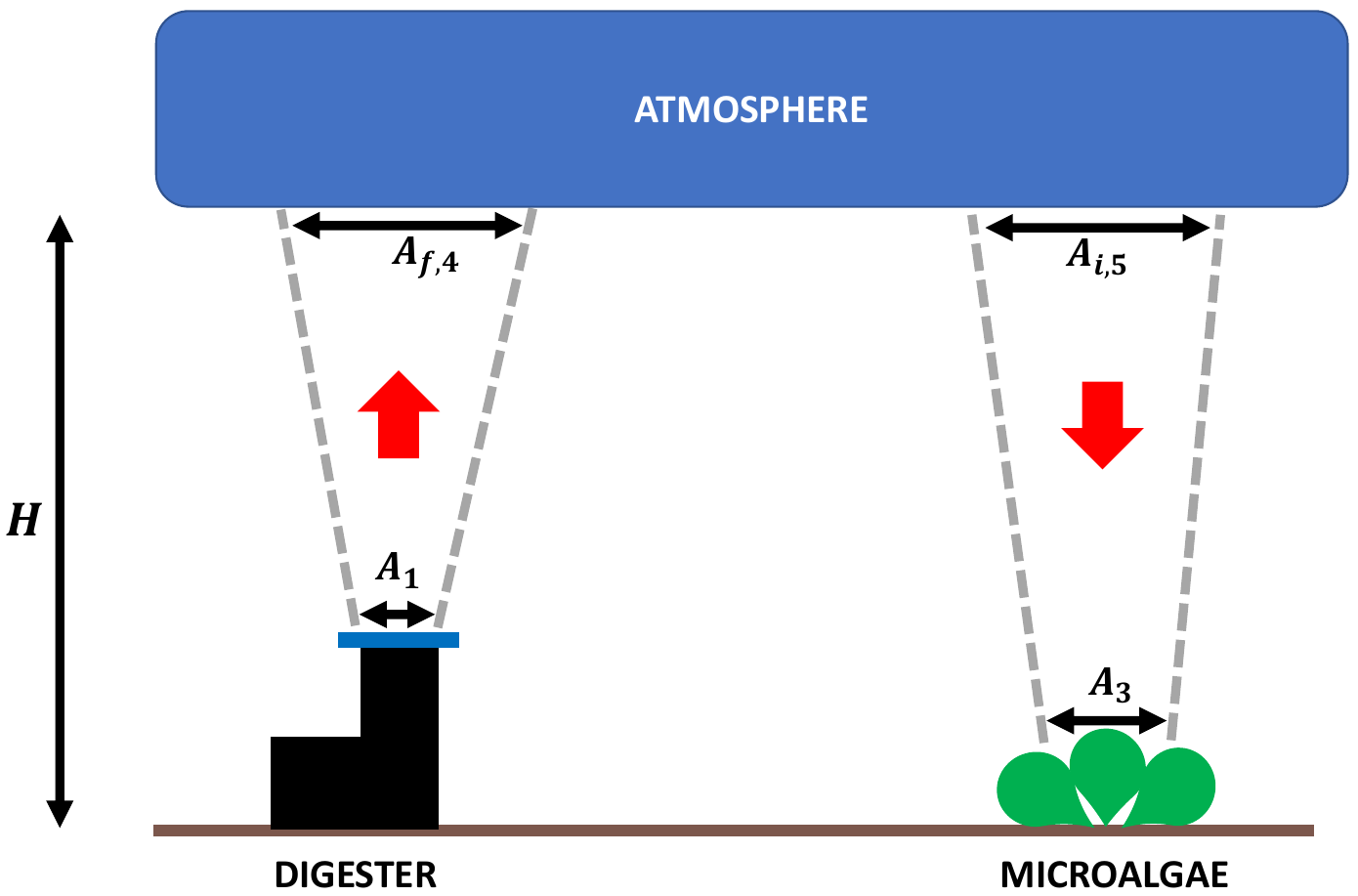}
\caption{Physical representation and geometry of the network to be designed.}
\label{fig:RealisticRepresentation}
\end{subfigure}
\begin{subfigure}{0.5\textwidth}
\centering
\includegraphics[width=0.3\textwidth]{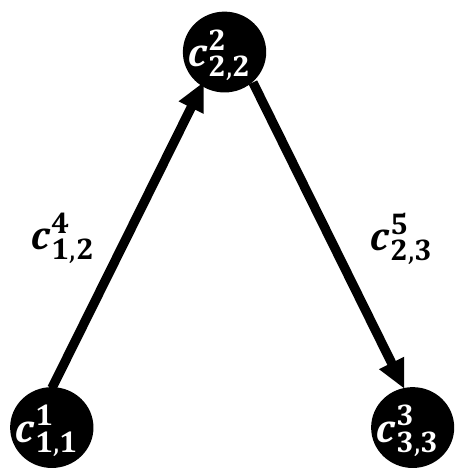}
\caption{Compartmental digraph corresponding to the system in (a).}
\label{fig:CompartmentalDigraph}
\end{subfigure}
\caption{Proposed carbon dioxide network design: Physical representation with geometry in (a) and its compartmental digraph in (b). The red arrows indicate the flow of carbon dioxide.}
\label{fig:TheNetwork}
\end{figure}

Hence, in this case and in the notation of \cite{zocco2023thermodynamical}, 
\begin{equation}
\mathcal{N} = \{c^1_{1,1}, c^2_{2,2}, c^3_{3,3}, c^4_{1,2}, c^5_{2,3}\} 
\end{equation}
with the compartmental digraph shown in Fig. \ref{fig:CompartmentalDigraph}. The atmosphere, the digester, and the cultivation are the vertex-compartments, while the two virtual ducts are the arc-compartments. Note that an alternative configuration could be achieved by connecting the digester directly with the microalgae so that $\text{CO}_2$ does not spread into the atmosphere.

\subsection{Step 3} In this paper, the dynamics of each compartment $c^k_{i,j} \in \mathcal{N}$ is derived from a mass balance. Firstly, the dynamical balances of each compartment are detailed. Then, the compartmental controller of the anaerobic digester (i.e., $c^1_{1,1}$) is developed.

\subsubsection{Anaerobic digester (i.e., $c^1_{1,1}$)} 
A mass balance of an anaerobic digester is given in \cite{bernard2001dynamical,campos2019hybrid}. The model considers two main reactions: the organic substrate $S_1(t)$ transforms into volatile fatty acids $S_2(t)$ through acidogenic bacteria $X_1(t)$, then the volatile fatty acids $S_2(t)$ transform into methane and $\text{CO}_2$ through methanogenic bacteria $X_2(t)$. The total inorganic carbon concentration $C(t)$ and the total alkalinity $Z(t)$ complete the system of six ODEs for the anaerobic digestion given by       
\begin{equation}
\begin{aligned}
\dot{X}_1(t) = {} & \left[\mu_1(\bm{x}(t)) - \alpha D_1(t)\right]X_1(t), \label{eq:digester1}\\
& X_1(0) = X_{1,0}, \quad t \geq 0,\\
\end{aligned}
\end{equation}
\begin{equation}
\begin{aligned}
\dot{X}_2(t) = {} & \left[\mu_2(\bm{x}(t)) - \alpha D_2(t)\right]X_2(t), \label{eq:digester2}\\
& X_2(0) = X_{2,0},
\end{aligned}
\end{equation}
\begin{equation}
\begin{aligned}
\dot{S}_1(t) = {} & D_3(t) \left(S_{1\text{in}} - S_1(t)\right) - k_1 \mu_1(\bm{x}(t))X_1(t), \label{eq:digester3}\\
& S_1(0) = S_{1,0},
\end{aligned}
\end{equation}
\begin{equation}
\begin{aligned}
\dot{S}_2(t) = {} & D_4(t) \left(S_{2\text{in}} - S_2(t)\right) + k_2 \mu_1(\bm{x}(t))X_1(t) \\
& - k_3 \mu_2(\bm{x}(t))X_2(t), \label{eq:digester4}\\
& S_2(0) = S_{2,0},
\end{aligned}
\end{equation}
\begin{equation}
\begin{aligned}
\dot{Z}(t) = {} & D_5(t) \left(Z_{\text{in}} - Z(t)\right), \label{eq:digester5}\\
& Z(0) = Z_0,
\end{aligned}
\end{equation}
\begin{equation}
\begin{aligned}
\dot{C}(t) = {} & D_6(t)\left(C_{\text{in}} - C(t)\right) - \dot{m}_{1,2}(\bm{x}(t)) \\
& + k_4\mu_1(\bm{x}(t))X_1(t) + k_5 \mu_2(\bm{x}(t))X_2(t), \label{eq:digester6} \\
& C(0) = C_0,
\end{aligned}
\end{equation}
where $\bm{x}(t) = \left[X_1(t), X_2(t), S_1(t), S_2(t), Z(t), C(t)\right]^\top$ is the system state,
\begin{equation}
\mu_1(\bm{x}(t)) = \mu_{1\text{max}} \frac{S_1(t)}{S_1(t) + K_{\text{S1}}},
\end{equation}
$D_j(t)|_{j = 1, 2, \dots, 6}$ is the dilution rate for the $j$-th state,
\begin{equation}
\mu_2(\bm{x}(t)) = \mu_{2\text{max}} \frac{S_2(t)}{S_2(t) + K_{\text{S2}} + (S_2(t)/K_{\text{I2}})^2}, 
\end{equation}
\begin{equation}\label{eq:CO2flow}
\begin{aligned}
\dot{m}_{1,2}(\bm{x}(t)) = {} & f_\text{r} \{k_{\text{L}}a\left[C(t) + S_2(t) - Z(t) \right. \\
& \left. - K_{\text{H}}P_\text{C}(\bm{x}(t))\right]\} \quad \left(\frac{\text{mmol}}{\text{Ld}}\right),
\end{aligned}
\end{equation}
and $\alpha, S_{1\text{in}}$, $S_{2\text{in}}$, $k_1$, $k_2$, $k_3$, $k_4$, $k_5$, $Z_{\text{in}}$, $C_{\text{in}}$, $\mu_{1\text{max}}$, $\mu_{2\text{max}}$, $K_{\text{S1}}$, $K_{\text{S2}}$, $K_{\text{I2}}$, $k_{\text{L}}a$, and $K_{\text{H}}$ are constant system parameters detailed in \cite{bernard2001dynamical,campos2019hybrid}. The $\text{CO}_2$ flow produced by the anaerobic digester is $\dot{m}_{1,2}(\bm{x}(t))$ and it is given in (\ref{eq:CO2flow}), where 
\begin{equation}
\begin{aligned}
P_\text{C}(\bm{x}(t)) = {} & \frac{\phi(\bm{x}(t))}{2K_{\text{H}}} \\ 
& - \frac{\sqrt{\phi^2(\bm{x}(t)) - 4 K_\text{H}P_{\text{T}} (C(t) + S_2(t) - Z(t))}}{2K_{\text{H}}} 
\end{aligned}
\end{equation} 
and
\begin{equation}
\begin{aligned}
\phi(\bm{x}(t)) = {} & C(t) + S_2(t) - Z(t) + K_{\text{H}}P_{\text{T}} \\ 
& + \frac{k_6}{k_{\text{L}}a}\mu_2(\bm{x}(t))X_2(t),
\end{aligned}
\end{equation}
with the constants $P_{\text{T}}$ and $k_6$ detailed in \cite{bernard2001dynamical,campos2019hybrid}. The multiplying factor $f_\text{r}$ takes into account the fact that, by UK law, anaerobic digestion plants must capture at least 80\% of carbon dioxide \cite{UKlawDigesters}. Therefore, $f_\text{r} \leq 0.2$. The carbon flow $\dot{m}_{1,2}(\bm{x}(t))$ produced by the digester flows inside the arc-compartment $c_{1,2}^4$ (modeled as a virtual duct of length $H$) and reaches the vertex-compartment $c^2_{2,2}$, i.e., the atmosphere.

\subsubsection{Microalgae (i.e., $c^3_{3,3}$)}
A well-known model of microalgae growth is given by Monod \cite{vatcheva2006experiment}, which can be applied to mass balances \cite{marcos2004output} (hence, it can be seen as a compartment of a TMN). In Monod's model, the growth rate $\mu$ of the culture depends on limiting factors that enter the model as multiplying terms \cite{solimeno2015new}. In this paper, we assume two limiting factors, namely, the $\text{CO}_2$ and the light intensity $I(t)$. Hence, in this case, Monod's model takes the form     
\begin{equation}
\begin{aligned}
\dot{X}_{\text{ALG}}(t) = {} & \mu(S, I)X_{\text{ALG}}(t) - \frac{1}{T_{\text{h}}}X_{\text{ALG}}(t), \label{eq:MonodEq1}\\
& X_{\text{ALG}}(0) = X_{\text{ALG},0}, \quad t \geq 0,
\end{aligned}
\end{equation}
\begin{equation}
\begin{aligned}
\dot{S}(t) = {} & \frac{1}{T_{\text{h}}}\left(S_{\text{in}} - S(t) \right) - \rho(S, I)X_{\text{ALG}}(t)  \\
& S(0) = S_0,
\end{aligned}
\end{equation}
where
\begin{equation}\label{eq:MonodGrowthRate}
\mu(S, I) = \mu_{\text{ALG}} \frac{I(t)}{I(t) + K_{\text{sI}} + \frac{I^2(t)}{K_{\text{iI}}}}\frac{S(t)}{S(t) + K_{\text{S}}} \\
\end{equation}
and
\begin{equation}
\rho(S, I) = \frac{1}{Y} \mu(S, I). \label{eq:MonodEq4}
\end{equation}
Here $X_{\text{ALG}}(t)$ [$\mu \text{m}^3/\text{L}$] is the total amount of biomass per unit volume (the volume can be converted into mass through the density), $S(t)$ [$\mu \text{mol}/\text{L}$] is the concentration of remaining nutrients, $\mu$ is the growth rate, $\rho$ is the rate of the nutrient consumption, $\mu_{\text{ALG}}$ is the maximum growth rate, $I(t)$ is the light intensity, $T_{\text{h}}$ is the hydraulic retention time, $K_{\text{S}}$ is the half-saturation constant for substrate uptake, $K_{\text{iI}}$ and $K_{\text{sI}}$ are two coefficients modeling the light influence on microalgae growth taken from \cite{bernard2011hurdles}, $Y$ is the growth yield, and $S_{\text{in}}$ is the input nutrient concentration \cite{vatcheva2006experiment}. The light intensity $I(t)$ affects the growth rate (\ref{eq:MonodGrowthRate}) according to the equation in \cite{bernard2011hurdles}. 

As we aim to study the atmospheric $\text{CO}_2$ absorption through properly-controlled microalgae, we need to introduce the carbon dioxide absorption into the model (\ref{eq:MonodEq1})-(\ref{eq:MonodEq4}). The photosynthesis is performed by microalgae to convert radiant energy into chemical energy of microalgae tissues, therefore it determines the microalgae growth. The process of photosynthesis can be written as \cite{hall1999photosynthesis}
\begin{equation}
\text{CO}_2 + \text{H}_2\text{O} \xrightarrow[\text{microalgae}]{\text{light}} [\text{CH}_2\text{O}] + \text{O}_2, 
\end{equation}
where the light can be both from the sun or from an artificial source \cite{darko2014photosynthesis}. Let $\rho_{\text{CO}_2}$ be the rate of carbon dioxide consumption, which is a fraction of the uptake of all the nutrients, i.e., a fraction of $\rho(S, I)$ given in (\ref{eq:MonodEq4}). This can be formulated as 
\begin{equation}\label{eq:carbonUptake}
\rho_{\text{CO}_2}(S, I) = K_{\text{CO}_2} \rho(S, I),  
\end{equation}
where $K_{\text{CO}_2} \in (0, 1)$ is a constant quantifying which fraction of $\rho$ is carbon dioxide. The carbon dioxide uptake (\ref{eq:carbonUptake}) corresponds to the flow rate of $\text{CO}_2$ from the vertex-compartment $c^2_{2,2}$ to $c^3_{3,3}$, that is,
\begin{equation}
\dot{m}_{2,3}(S, I) = \rho_{\text{CO}_2}(S, I) \quad \left(\frac{\mu\text{mol}}{\mu\text{m}^3 \text{d}}\right).  
\end{equation}   
\begin{remark}
Along with the photosynthesis, microalgae carry out the process of respiration, which releases $\text{CO}_2$ to the atmosphere. Typically, the absorption of $\text{CO}_2$ in green plants is greater than its production \cite{solimeno2015new,hall1999photosynthesis}. We account for respiration by choosing a lower value of $K_{\text{CO}_2}$ rather than considering both the $\text{CO}_2$ absorption and production flows. For this reason, the flow in compartment $c^5_{2,3}$ is directed only from $c^2_{2,2}$ to $c^3_{3,3}$ (see Fig. \ref{fig:CompartmentalDigraph}). The respiration generates a flow in the opposite direction.       
\end{remark}

\subsubsection{Virtual ducts (i.e., $c^4_{1,2}$ and $c^5_{2,3}$)} 
We assume the realistic scenario in which the atmosphere is immediately above the digester and the microalgae, which corresponds to the case $H \approx 0$ in Fig. \ref{fig:RealisticRepresentation}. Hence, we have that
\begin{equation}
\lim_{H \to 0} A_{f,4} = A_1 \quad \text{and} \quad \lim_{H \to 0} A_{i,5} = A_3.  
\end{equation}
Since the virtual ducts have an infinitesimal length, their effect on the fluid is negligible.

\subsubsection{Atmosphere (i.e., $c^2_{2,2}$)}
The mass balance of $\text{CO}_2$ involves the carbon dioxide released by the anaerobic digestion plant and the microalgae absorption, that is, 
\begin{equation}\label{eq:AtmAccRate}
\begin{aligned}
\frac{\text{d}m_2(t)}{\text{d}t} = {} 
& \dot{m}_{1,2}(t) - \dot{m}_{2,3}(t) \\
& = f_\text{r}\{k_{\text{L}}a\left[C(t) + S_2(t) - Z(t) - K_{\text{H}}P_\text{C}\right]\} \\
& - K_{\text{CO}_2}\frac{\mu_{\text{ALG}}}{Y} \frac{I(t)}{I(t) + K_{\text{sI}} + \frac{I^2(t)}{K_{\text{iI}}}}\frac{S(t)}{S(t) + K_{\text{S}}}.
\end{aligned}  
\end{equation}

\subsubsection{Digester controller}
The products of the anaerobic digestion are mainly two gases: methane and carbon dioxide. As the target material of this network is carbon dioxide (as chosen in Section \ref{subsec:Step1}), we design a continuous-time, initial-condition-dependent optimal control to regulate the $\text{CO}_2$ output flow. The control input is the dilution rate $D(t)$. 

To develop our control architecture, we first recall a theorem from \cite{haddad2015finite}, which provides an architecture for an optimal finite-time stabilizing controller for affine dynamical systems of the form
\begin{equation}\label{eq:affineForm}  
\dot{\bm{x}}(t) = \bm{f}(\bm{x}(t)) + \bm{G}(\bm{x}(t))\bm{u}(t), \quad \bm{x}(0) = \bm{x}_0, \quad t \geq 0,
\end{equation}
where, for $t \geq 0$, $\bm{x}(t) \in \mathbb{R}^{n}, \bm{u}(t) \in \mathbb{R}^{l}$, and $\bm{f}:\mathbb{R}^{n} \rightarrow \mathbb{R}^{n}$ and $\bm{G}: \mathbb{R}^{n} \rightarrow \mathbb{R}^{n \times l}$ are such that $\bm{f}(\cdot)$ and $\bm{G}(\cdot)$ are continuous in $\bm{x}$ and $\bm{f}(0) = \bm{0}$. 
\begin{theorem}[\cite{haddad2015finite}]\label{th:baseline}
Consider the affine nonlinear dynamical system (\ref{eq:affineForm}). Assume that there exist a continuously differentiable, radially unbounded function $V : \mathbb{R}^{n} \rightarrow \mathbb{R}$ and real numbers $c > 0$ and $\beta \in (0, 1)$ such that the following conditions hold:
\begin{gather}
V(0) = 0, \label{eq:cond1}\\ 
V(\bm{x}) > 0, \quad \bm{x} \in \mathbb{R}^{n} \setminus \{0\}, \label{eq:cond2}\\ 
V^\prime(\bm{x}) \left[\bm{f}(\bm{x}) - \frac{1}{2}\bm{G}(\bm{x})\bm{R}^{-1}_2(\bm{x})\bm{L}^\top_2(\bm{x}) - \right. \notag\\ 
\left. \frac{1}{2} \bm{G}(\bm{x})\bm{R}^{-1}_2(\bm{x})\bm{G}^\top(\bm{x}) {V^{\prime}}^\top(\bm{x})\right] \leq \notag\\
- c(V(\bm{x}))^\beta, \quad \bm{x} \in \mathbb{R}^{n}, \label{eq:cond3}\\ 
\bm{L}_2(0) = 0, \label{eq:cond4} 
\end{gather}
where $\bm{L}_2 : \mathbb{R}^{n} \rightarrow \mathbb{R}^{1 \times l}$ is continuous on $\mathbb{R}^{n}$ and $\bm{R}_2(\bm{x}) > 0$ is continuous on $\mathbb{R}^{n}$. Then, with the feedback control 
\begin{equation}\label{eq:uOld}
\bm{u}(t) = \bm{\phi}(\bm{x}) = - \frac{1}{2}\bm{R}^{-1}_2(\bm{x})\left[\bm{L}_2(\bm{x}) + V^\prime(\bm{x})\bm{G}(\bm{x})\right]^\top,   
\end{equation}
the zero solution $\bm{x}(t) \equiv 0$ to the affine dynamical system (\ref{eq:affineForm}) is globally finite-time stable. Moreover, there exists a settling-time function $T : \mathbb{R}^{n} \rightarrow [0, \infty)$ such that
\begin{equation}\label{eq:settTimeOriginal}
T(\bm{x}_0) \leq \frac{1}{c(1-\beta)} (V(\bm{x}_0))^{1-\beta}, \quad \bm{x}_0 \in \mathbb{R}^{n}, 
\end{equation}
and the performance functional 
\begin{gather}
J(\bm{x}_0, \bm{u}(\cdot)) = \int_0^\infty \left[L_1(\bm{x}) + \bm{L}_2(\bm{x})\bm{u}(t) \right. \notag\\ 
\left. + \bm{u}^\top(t)\bm{R}_2(\bm{x})\bm{u}(t)\right] \textup{d}t \label{eq:JOld}
\end{gather}
is minimized, with
\begin{equation}\label{eq:L1Old}
L_1(\bm{x}) = \bm{\phi}^\top(\bm{x})\bm{R}_2(\bm{x})\bm{\phi}(\bm{x}) - V^\prime(\bm{x})\bm{f}(\bm{x}).
\end{equation}
\end{theorem}

To verify conditions (\ref{eq:cond1})-(\ref{eq:cond4}) for the six-state system (\ref{eq:digester1})-(\ref{eq:digester6}) of the digester, we adopted the expressions of $\bm{L}_2(\bm{x})$ and $\bm{R}^{-1}_2(\bm{x})$ proposed in \cite{zocco2023thermodynamical} (with $l$ = $n$) given by
\begin{equation}\label{L2choice}
\bm{L}_2(\tilde{\bm{x}}) = 2\left[\bm{f}^\top(\tilde{\bm{x}})\bm{G}(\tilde{\bm{x}})\right]
\end{equation}
and
\begin{equation}\label{eq:R2choice}
\bm{R}^{-1}_2(\tilde{\bm{x}}) = \bm{G}^{-1}(\tilde{\bm{x}})(\bm{G}^\top(\tilde{\bm{x}}))^{-1},
\end{equation}
where $\tilde{\bm{x}} \in \mathbb{R}^{n}$ is the system state translated to have the zero equilibrium solution $\tilde{\bm{x}}(t) \equiv \bm{0}$ corresponding to the desired operating equilibrium point ``SS6'' in \cite{campos2019hybrid}. With these choices for $\bm{L}_2(\tilde{\bm{x}})$ and $\bm{R}^{-1}_2(\tilde{\bm{x}})$, (\ref{eq:cond3}) simplifies as
\begin{equation}
-\frac{1}{2}V^\prime(\bm{x}) {V^\prime}^\top(\bm{x}) \leq -c(V(\bm{x}))^\beta,
\end{equation}
which is independent of $\bm{f}(\bm{x})$ and $\bm{G}(\bm{x})$. Moreover, (\ref{eq:cond4}) is satisfied regardless of the expressions of $\bm{f}(\bm{x})$ and $\bm{G}(\bm{x})$ since $\bm{f}(0) = \bm{0}$. However, this choice of $\bm{R}^{-1}_2(\tilde{\bm{x}})$ requires that $l = n$ and $\bm{G}^{-1}(\tilde{\bm{x}})$ exists. 
\begin{theorem}[Initial-condition-dependent control]\label{th:inCoDepCon}
Consider the affine nonlinear dynamical system (\ref{eq:affineForm}) with $l$ = $n$. Assume that there exists a continuously differentiable, radially unbounded function $V : \mathbb{R}^{n} \rightarrow \mathbb{R}$ of the form
\begin{equation}\label{eq:VofxChoice}
V(\bm{x}, \bm{x}_0, T_{\textup{max}}) = p(\bm{x}^\top\bm{x})^q,
\end{equation}
where
\begin{equation}\label{eq:pofTmax}
p(\bm{x}_0, T_{\textup{max}}) = \frac{1}{2}\left(\bm{x}^\top_0\bm{x}_0\right)^{\frac{1}{1+T_{\textup{max}}}} \left(\frac{1+T_{\textup{max}}}{T_{\textup{max}}}\right)^2,
\end{equation}
\begin{equation}\label{eq:qofTmax}
q(T_{\textup{max}}) = \frac{T_{\textup{max}}}{1 + T_{\textup{max}}},
\end{equation} 
and $T_{\textup{max}}$ is an upper bound of the settling-time $T$, i.e. $T \leq T_{\textup{max}}$. 
Then, with the feedback control 
\begin{equation}\label{eq:uNew}
\begin{aligned}
\bm{u}(t) = {} & \bm{\phi}(\bm{x}, \bm{x}_0) \\
& = -\frac{1}{2}\bm{G}^{-1}(\bm{x})\left[2\bm{f}(\bm{x}) + {V^\prime}^\top(\bm{x}, \bm{x}_0)\right],
\end{aligned}
\end{equation}
the zero solution $\bm{x}(t) \equiv 0$ to the affine dynamical system (\ref{eq:affineForm}) is globally fixed-time stable with settling-time $T \leq T_{\textup{max}}$. 
Moreover, the performance functional 
\begin{equation}
\begin{aligned}
J(\bm{x}_0, \bm{u}(\cdot)) = {} & \int_0^\infty \left[L_1(\bm{x}(t)) \right. \\
& \left. + 2\bm{f}^\top(\bm{x}(t))\bm{G}(\bm{x}(t))\bm{u}(t) \right. \\ 
& \left. + \bm{u}^\top(t)\bm{G}^\top(\bm{x}(t))\bm{G}(\bm{x}(t))\bm{u}(t)\right] \textup{d}t \label{eq:Jnew}
\end{aligned}
\end{equation}
is minimized, with
\begin{equation}\label{eq:L1new}
\begin{aligned}
L_1(\bm{x}, \bm{x}_0) = {} & \bm{\phi}^\top(\bm{x})\bm{G}^\top(\bm{x})\bm{G}(\bm{x})\bm{\phi}(\bm{x}) \\ 
& - V^\prime(\bm{x}, \bm{x}_0)\bm{f}(\bm{x}).
\end{aligned}
\end{equation}
Furthermore, in this case the closed-loop dynamics reduce to
\begin{equation}\label{eq:closedLoop}
\dot{\bm{x}}(t) = -\frac{1}{2}{V^\prime}^\top(\bm{x}(t), \bm{x}_0, T_{\textup{max}}).
\end{equation}
\end{theorem}      

\begin{proof}
First, note that with $c = (V(x_0))^{1-\beta}$, $V(\bm{x})$ given by (\ref{eq:VofxChoice}), and $\bm{L}_2(\bm{x})$ given by (\ref{L2choice}), (\ref{eq:cond1}), (\ref{eq:cond2}), and (\ref{eq:cond4}) are satisfied.    Hence, only condition (\ref{eq:cond3}) remains to be satisfied. To show this, note that with $\bm{R}^{-1}_2(\bm{x})$ given by (\ref{eq:R2choice}),  (\ref{eq:cond3}) reduces to  
\begin{equation}\label{eq:cond33reduced}
-\frac{1}{2}V^\prime(\bm{x}){V^\prime}^\top(\bm{x}) \leq - c(V(\bm{x}))^\beta, \quad \bm{x} \in \mathbb{R}^{n}.
\end{equation}
Now, taking (\ref{eq:settTimeOriginal}) as an equality with $c = (V(x_0))^{1-\beta}$ yields
\begin{equation}\label{eq:TmaxEquality}
T_{\textup{max}} = \frac{1}{1-\beta} \rightarrow \beta = 1 - \frac{1}{T_{\textup{max}}}. 
\end{equation}
With $c = (V(x_0))^{1-\beta}$, (\ref{eq:VofxChoice}), (\ref{eq:cond33reduced}), and (\ref{eq:TmaxEquality}) give
\begin{equation}\label{eq:equalityWithpqTmax}
2q^2p^{\frac{1}{q}}(V(\bm{x}))^{2 - \frac{1}{q}} = (V(\bm{x}_0))^{\frac{1}{T_{\textup{max}}}} (V(\bm{x}))^{1-\frac{1}{T_{\textup{max}}}}.
\end{equation}
Note that each side of (\ref{eq:equalityWithpqTmax}) can be divided into a state-dependent term and a state-independent term. Equating the state-dependent terms yields
\begin{equation}
2 - \frac{1}{q} = 1 - \frac{1}{T_{\textup{max}}},
\end{equation}
which leads to (\ref{eq:qofTmax}). Finally, inserting (\ref{eq:qofTmax}) into (\ref{eq:equalityWithpqTmax}) and equating the state-independent terms we obtain $p(\bm{x}_0, T_{\textup{max}})$ in (\ref{eq:pofTmax}). Equations (\ref{eq:uNew}), (\ref{eq:Jnew}), and (\ref{eq:L1new}) follow by inserting (\ref{L2choice}) and (\ref{eq:R2choice}) into (\ref{eq:uOld}), (\ref{eq:JOld}), and (\ref{eq:L1Old}), respectively, while (\ref{eq:closedLoop}) follows by inserting (\ref{eq:uNew}) into (\ref{eq:affineForm}).
\end{proof}

\begin{remark}
The initial-condition-dependent control can be seen as a framework which identifies a fixed-time optimal controller for each initial conditions $\bm{x}_0$. It has resulted from addressing a limitation of the finite-time formulation \cite{haddad2015finite} in which the settling-time depends on $\bm{x}_0$ (see (\ref{eq:settTimeOriginal})). Hence, choosing $c = (V(x_0))^{1-\beta}> 0$, the initial condition now appears explicitly in $p = p(\bm{x}_0)$ and thus in $\bm{\phi}(\bm{x}, \bm{x}_0)$. The advantage is that now the upper bound of the settling-time $T_{\textup{max}}$ is a tuning parameter independent of $\bm{x}_0$. However, it is important to note that (\ref{eq:uNew}) is not a pure state feedback controller since $\bm{x}_0$ enters $\bm{\phi}(\bm{x}(t))$ for $t \geq 0$; in contrast, for a pure state feedback control \cite{haddad2015finite} $\bm{x}_0$ enters $\bm{\phi}(\bm{x}(t))$ only for $t = 0$ when $\bm{\phi}(\bm{x}(0)) = \bm{\phi}(\bm{x}_0)$.      
\end{remark}

\begin{corollary}\label{cor:Tmax}
$T_{\textup{max}} \in (1, \infty)$.
\end{corollary}

\begin{proof}
The result is immediate by noting that $T_{\text{max}} = \frac{1}{1 - \beta}$, where $\beta \in (0, 1)$.
\end{proof}

\section{Numerical Study}\label{sec:numStudy}
This section analyses the network using the numerical values listed in Table \ref{tab:valuesOfParam}. 
\begin{table}
\centering
\caption[Caption for LOF]{Values of the parameters used in the numerical study. Missing values are derived from the equations. The values are collected in the paper source code\footnotemark[\value{footnote}].}
\label{tab:valuesOfParam}
\renewcommand{\arraystretch}{2}
\begin{tabular}{cc} 
\hline
Compartment & Parameter values\\ 
\hline
\hline
\multirow{4}{*}{$c^1_{1,1}$} & $\tilde{\bm{x}}_{0,1} = [-1, 0.5, 1.0, 1.5, 0.8, -0.5]^\top$ \\
& $\tilde{\bm{x}}_{0,2} = [-1.5, 1.25, 0.4, 1.8, -1.8, -2.2]^\top$ \\
& $f_\text{r} = 0.15$ \\
& Remaining values taken from \cite{bernard2001dynamical,campos2019hybrid} \\
\hline 
\multirow{2}{*}{$c^3_{3,3}$} & $K_{\text{CO}_2} = 0.3$ \\
& Remaining values taken from \cite{vatcheva2006experiment,bernard2011hurdles} \\
\hline
\end{tabular}
\renewcommand{\arraystretch}{1}
\end{table}
The anaerobic digester is stabilized to the equilibrium ``SS6'' in \cite{campos2019hybrid} using the initial-condition-dependent controller (Theorem \ref{th:inCoDepCon}). The trajectories of the six state variables are shown in Fig. \ref{fig:DigesterStates}. Specifically, Fig. \ref{fig:DigesterStates_T_MAXchanges} shows that the origin (i.e., the equilibrium ``SS6'') is reached before the upper bound $T_{\text{max}}$ by all the state variables. In the case of $T_{\text{max}} = 3.5$ days, the origin is reached in approximately 0.5 days before $T_{\text{max}}$, whereas with $T_{\text{max}} = 1.1$ days the equilibrium is reached in the proximity of $T_{\text{max}}$. Fig. \ref{fig:DigesterStates_x0changes} considers the situation in which $T_{\text{max}}$ is fixed ($T_{\text{max}} = 3.5$) while the initial conditions change. In this case, the origin is reached by all the states in approximately $ 3 < T_{\text{max}}$ days regardless of $\tilde{\bm{x}}_0$.  
\begin{figure}
\begin{subfigure}{.24\textwidth}
  \centering
  \includegraphics[width=1\linewidth]{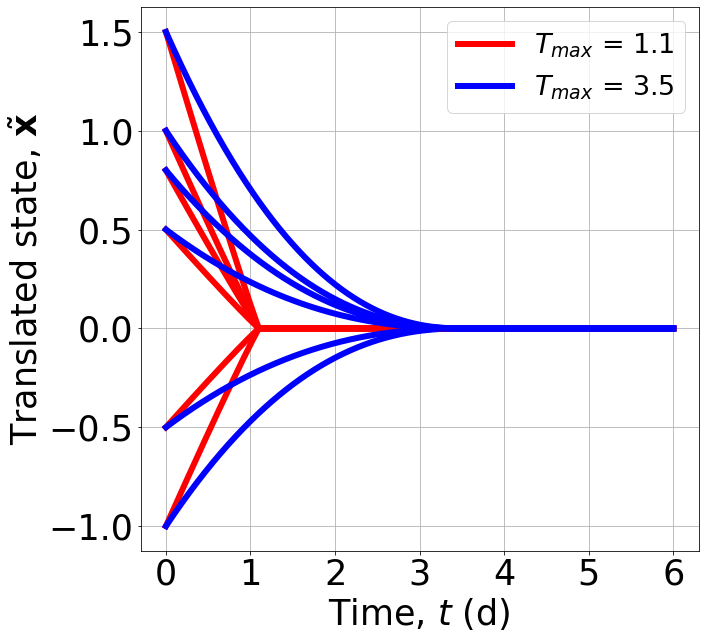}
  \caption{With $\tilde{\bm{x}}_0$ fixed}
  \label{fig:DigesterStates_T_MAXchanges}
\end{subfigure}%
\begin{subfigure}{.24\textwidth}
  \centering
  \includegraphics[width=1\linewidth]{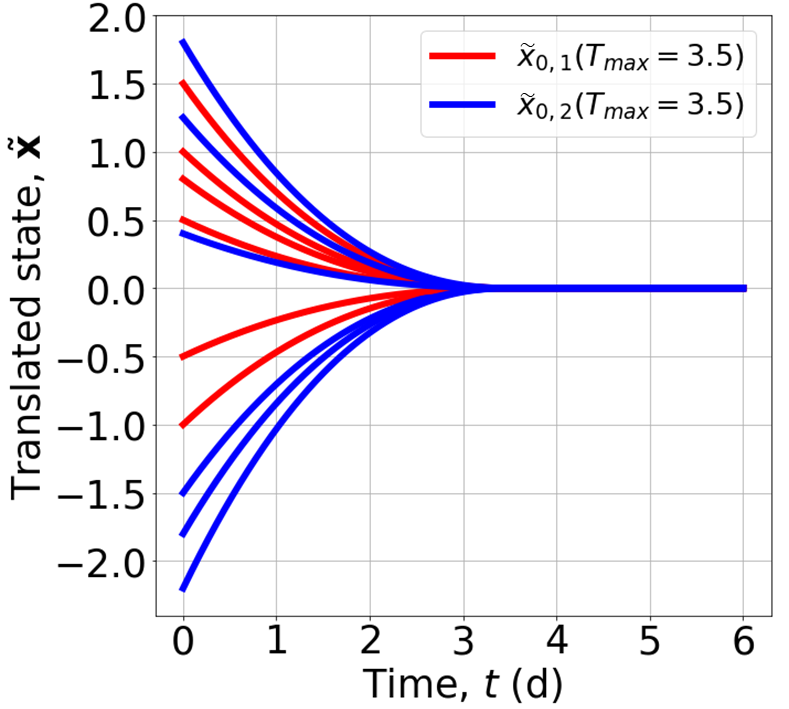}
  \caption{With $T_{\text{max}}$ fixed}
  \label{fig:DigesterStates_x0changes}
\end{subfigure}
\caption{Components of the state of $c^1_{1,1}$ vs. time with the initial-condition-dependent controller (Theorem \ref{th:inCoDepCon}).}
\label{fig:DigesterStates}
\end{figure}
    
Figure \ref{fig:FlowsAlgaeOpenLoop} shows the dynamics of compartments $c^2_{2,2}$, $c^3_{3,3}$, $c^4_{1,2}$, and $c^5_{2,3}$. Specifically, Fig. \ref{fig:Algal_biomass} shows that the algal biomass $X_{\text{ALG}}(t)$ has a first-order dynamics; the equilibrium value is approximately 50 $\mu \text{m}^3$/L (the volume can be converted into mass through the density) and it is reached in 12 days. The flow rate of $\text{CO}_2$ released by the anaerobic digester decreases during the first 3 days as seen in Fig. \ref{fig:m12_dot}, and then converges to 175 mmol$\text{L}^{-1}\text{d}^{-1}$ (note that this value is highly affected by $f_\text{r}$, which quantifies the digester capability to capture its emissions). Figure \ref{fig:m23_dot} shows the microalgae uptake of $\text{CO}_2$ as a nutrient; their absorption rate converges to 0.28 $\mu \text{mol} \mu \text{m}^{-3}\text{d}^{-1}$ after 8 days with a peak of 0.52 $\mu \text{mol} \mu \text{m}^{-3}\text{d}^{-1}$ in $t < 1$ days. The most important result is the accumulation rate of atmospheric carbon dioxide shown in Fig. \ref{fig:m2_accumulRate}, which depends on the adversarial interaction between the carbon digester outflow and the microalgae absorption (see (\ref{eq:AtmAccRate})). 
\begin{remark}
The units used for $\dot{m}_{1,2}$ and $\dot{m}_{2,3}$ are those used by the references that developed the digester and the microalgae dynamical models as indicated in Table \ref{tab:valuesOfParam}. Specifically, the unit is mmol/Ld for the former and $\mu$mol/$\mu \text{m}^3$d for the latter. While they appear to be different, the reader can verify that they are equivalent.   
\end{remark}
Since the microalgae uptake is three orders of magnitude smaller than the digester emissions, it follows that $\frac{\text{d}m_2(t)}{\text{d}t} \approx \dot{m}_{1,2}(t)$.          
\begin{figure*}[t]
\begin{subfigure}{.24\textwidth}
  \centering
  \includegraphics[width=1\linewidth]{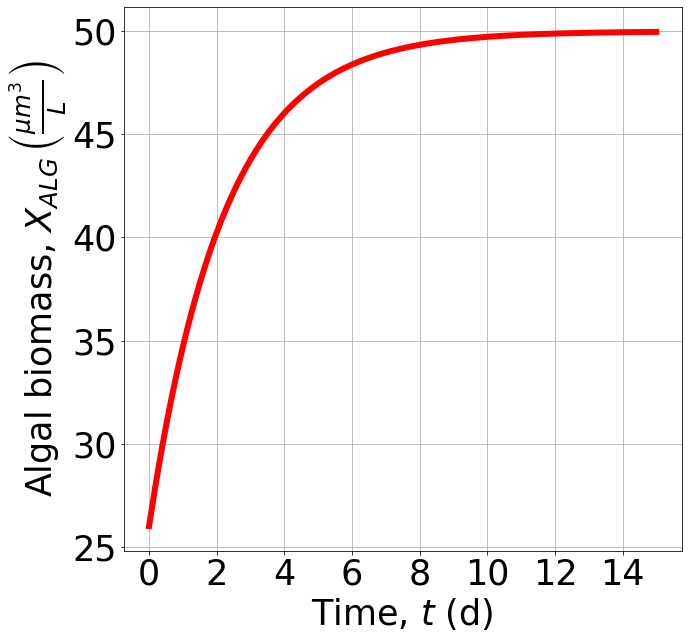}
  \caption{Algal biomass in $c^3_{3,3}$}
  \label{fig:Algal_biomass}
\end{subfigure}%
\begin{subfigure}{.24\textwidth}
  \centering
  \includegraphics[width=1\linewidth]{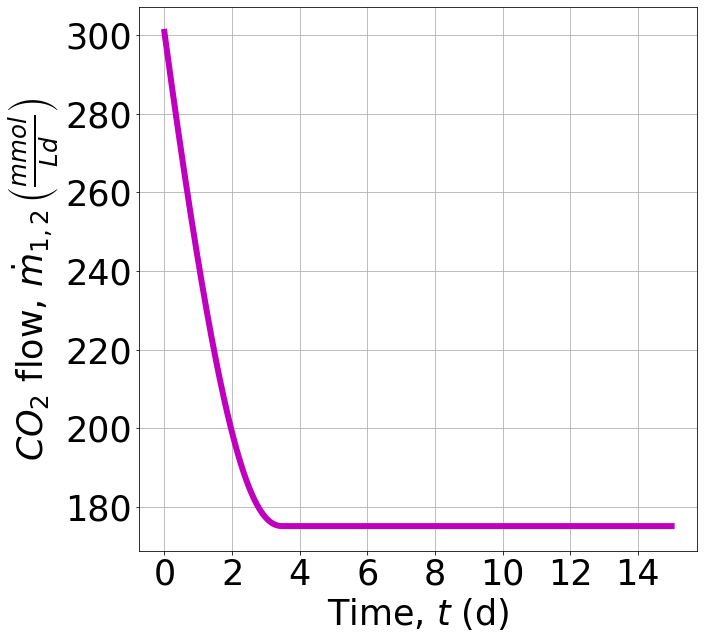}
  \caption{Digester carbon outflow}
  \label{fig:m12_dot}
\end{subfigure} 
\begin{subfigure}{.24\textwidth}
  \centering
  \includegraphics[width=1\linewidth]{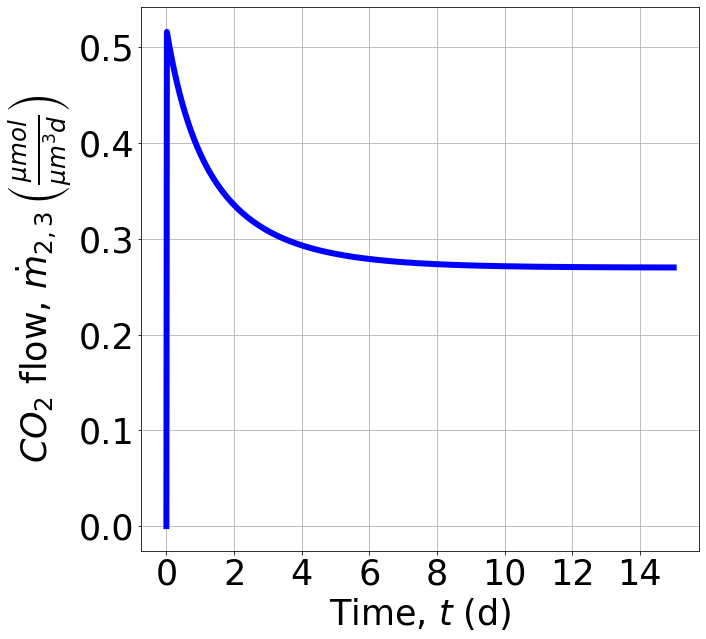}
  \caption{Microalgae carbon uptake}
  \label{fig:m23_dot}
\end{subfigure}
\begin{subfigure}{.24\textwidth}
  \centering
  \includegraphics[width=1\linewidth]{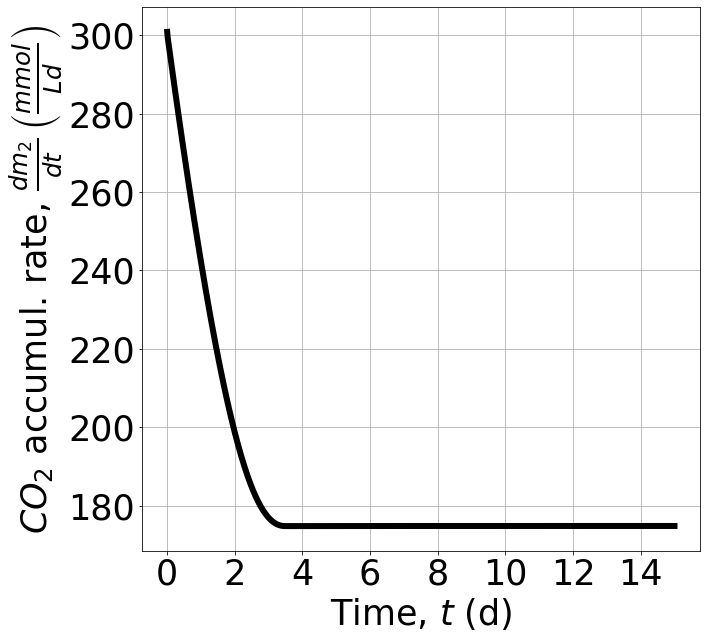}
  \caption{Accumulation rate in $c^2_{2,2}$}
  \label{fig:m2_accumulRate}
\end{subfigure}
\caption{Dynamics of compartments $c^2_{2,2}$, $c^3_{3,3}$, $c^4_{1,2}$, and $c^5_{2,3}$ using the values in Table \ref{tab:valuesOfParam}.}
\label{fig:FlowsAlgaeOpenLoop}
\end{figure*}
Note that the terms in (\ref{eq:AtmAccRate}) are per unit of volume. Hence, by indicating with $V_\text{d}$ the volume of the anaerobic digester and with $V_\text{m}$ the volume of the microalgae cultivation, using (\ref{eq:AtmAccRate}) we can calculate that the volume of the microalgae cultivation required to completely compensate the emissions of each liter of the digester at the steady state (i.e., $V_{\text{d}} = 1$ L) is
\begin{equation}
0 = \overline{\dot{m}}_{1,2} V_\text{d} - \overline{\dot{m}}_{2,3} V_\text{m}, 
\end{equation}
which gives
\begin{equation}\label{eq:ratioOfVolums}
V_\text{m} = \frac{\overline{\dot{m}}_{1,2}}{\overline{\dot{m}}_{2,3}}V_\text{d} = 625 \,\, \text{L},
\end{equation}    
where $\overline{\dot{m}}_{i,j}$ is the steady state of $\dot{m}_{i,j}$. In other words, the anaerobic digester releases to the atmosphere an amount of carbon dioxide that, to be compensated, requires a microalgae cultivation with a volume 625 times bigger than the volume of the digester. This is essentially because the microalgae concentration is very low; another possible factor is the limitation on the light intensity.

\textbf{Circularity analysis:} Next, we address how the circularity $\lambda(\mathcal{N})$ in equation (\ref{eq:circProblem}) is affected. Without microalgae, the circularity yields 
\begin{equation}
\lambda_\text{a} = - \dot{m}_{\text{f,c}}\Delta = - \dot{m}_{1,2}\Delta,
\end{equation}
where $\dot{m}_{\text{f,c}} \geq 0$ is the \emph{net finite-time sustainable} flow and $\Delta = 1$ day \cite{zocco2024circular}. In contrast, with the microalgae, the circularity becomes 
\begin{equation}
\lambda_\text{b} = - \dot{m}_{\text{f,c}}\Delta = - (\dot{m}_{1,2} - \dot{m}_{2,3})\Delta. 
\end{equation}
Since $\lambda_\text{b} > \lambda_\text{a}$, the microalgae increase the circularity of carbon dioxide. In particular, from (\ref{eq:ratioOfVolums}) it follows that a microalgae cultivation with a volume 625 times that of the digester yields $\lambda_\text{b} = 0$, which corresponds to the net zero target \cite{NZdefinition}.

\section{Towards Microalgae Control via Reinforcement Learning}\label{sec:RLcontrol}
In this section, we address the question whether it is possible to increase the carbon uptake of the microalgae, and hence, reduce the above volume factor of 625? To address this question, we needed a control strategy for the microalgae system (\ref{eq:MonodEq1})-(\ref{eq:MonodEq4}) that is able to maximize the carbon uptake, i.e., equation (\ref{eq:carbonUptake}). By choosing the light intensity $I(t)$ as the control input, the system (\ref{eq:MonodEq1})-(\ref{eq:MonodEq4}) is nonaffine in the control, and hence, the control laws considered above, i.e., Theorems \ref{th:baseline} and \ref{th:inCoDepCon}, cannot be used.

Thus, we approached this optimal control problem using RL by considering $I(t)$ as the action to be generated by the RL controller. Since we need to maximize the carbon uptake (\ref{eq:carbonUptake}), we set (\ref{eq:carbonUptake}) as the reward. First, we implemented the Monod's model (\ref{eq:MonodEq1})-(\ref{eq:MonodEq4}) using a state-of-art library for RL, namely, Stable-Baselines3 (SB3) \cite{stable-baselines3}. Specifically, we numerically solved the nonlinear ordinary differential equations (\ref{eq:MonodEq1})-(\ref{eq:MonodEq4}) using Euler's method. To tune the step size for Euler's method, namely, $\delta_\text{t}$, we reduced its value until the numerical solution given by Euler's method matched the one given by the Python solver \emph{scipy.integrate.odeint()} for the same initial conditions and the same input $I(t)$. This yielded $\delta_\text{t} = 0.00005$. A summary of the RL problem set-up is given in Table \ref{tab:RLsetup}.
\begin{table*}
\centering
\caption{RL problem set-up for microalgae control via light intensity $I(t)$.}
\label{tab:RLsetup}
\begin{tabular}{ccccc} 
\hline
Environment model & Integration method & $\delta_\text{t}$ & States & Observations \vspace{0.1cm}\\
Equations (\ref{eq:MonodEq1})-(\ref{eq:MonodEq4}) & Euler's & 0.00005 & $X_{\text{ALG}}(t)$, $S(t)$ & $X_{\text{ALG}}(t)$, $S(t)$ \vspace{0.5cm}\\
\hline
Actions & Initial conditions & Reward & Max episode length & Termination conditions \vspace{0.1cm}\\
$I(t)$ & \makecell{Randomly sampled from \\ uniform distributions} & Equation (\ref{eq:carbonUptake}) & 200 time steps & None \\
\hline
\end{tabular}
\end{table*}

The implementation of the Monod's model in SB3 enabled us to test the RL algorithms available in SB3 compatible with continuous-time actions and observations. We tested eight algorithms for training the RL controller, namely, the A2C \cite{mnih2016asynchronous}, the augmented random search (ARS) \cite{NEURIPS2018_7634ea65}, the CrossQ \cite{bhattcrossq}, the deep deterministic policy gradient (DDPG) \cite{lillicrap2019continuouscontroldeepreinforcement}, the proximal policy optimization (PPO) \cite{schulman2017proximal},  the soft actor-critic (SAC) \cite{haarnoja2018soft}, the truncated quantile critics (TQC) \cite{kuznetsov2020controlling}, and the trust region policy optimization (TRPO) \cite{pmlr-v37-schulman15}. The algorithms are summarized in Table \ref{tab:RLtunings} along with their tuning. Each algorithm trained the microalgae controller for 200,000 time steps with a maximum episode length of 200 time steps and with no termination conditions. At the end of each episode, the new initial conditions were randomly sampled from uniform distributions. All the executions were performed in Google Colaboratory (aka Colab) using an NVIDIA Tesla T4 GPU. 
\begin{table*}
\centering
\caption{Tunings of the RL algorithms. Except for the policy models, the tunings are the default settings implemented on Stable-Baselines3 \cite{stable-baselines3}.}
\label{tab:RLtunings}
\begin{tabular}{cc} 
\hline
A2C & \makecell{Policy model: multilayer perceptron; learning rate: 0.0007; number of steps: 5; discount factor: 0.99; entropy coefficient for the \\ loss calculation: 0; value function coefficient for the loss calculation: 0.5; optimizer: RMSprop.} \\
\hline
ARS & \makecell{Policy model: multilayer perceptron; number of random perturbations of the policy to try at each update step: 8; learning rate: 0.02; \\ exploration noise: 0.05; the passed policy has the weights zeroed before training: yes; alive bonus offset: 0; number of episodes \\ to evaluate each candidate: 1.} \\
\hline
CrossQ & \makecell{Policy model: multilayer perceptron; learning rate: 0.001; size of the replay buffer: 1000000; number of steps of the model to collect \\ transitions for before learning starts: 100; batch size: 256; discount factor: 0.99; number of steps to update the model: 1; \\ entropy regularization coefficient learned automatically: yes; policy delay: 3.}\\
\hline
DDPG & \makecell{Policy model: multilayer perceptron; learning rate: 0.001; size of the replay buffer: 1000000; number of steps of the model to collect \\ transitions for before learning starts: 100; batch size: 256; soft update coefficient: 0.005;  \\ discount factor: 0.99; optimizer: Adam; number of steps to update the model: 1.} \\
\hline
PPO & \makecell{Policy model: multilayer perceptron; learning rate: 0.0003; number of steps to update: 2048; batch size: 64; number of epochs when \\ optimizing the surrogate loss: 10; discount factor: 0.99; entropy coefficient for the loss calculation: 0; value function \\ coefficient for the loss calculation: 0.5; maximum value for the gradient clipping: 0.5; use of generalized state dependent \\ exploration (gSDE): no.}\\
\hline
SAC & \makecell{Policy model: multilayer perceptron; learning rate: 0.0003; size of the replay buffer: 1000000;  number of steps of the model to collect \\ transitions for before learning starts: 100; batch size: 256; soft update coefficient: 0.005; discount factor: 0.99; number of steps \\ to update the model: 1; entropy regularization coefficient learned automatically: yes.}\\
\hline
TQC &  \makecell{Policy model: multilayer perceptron; learning rate: 0.0003; size of the replay buffer: 1000000; number of steps of the model to collect \\ transitions for before learning starts: 100; batch size: 256; soft update coefficient: 0.005; discount factor: 0.99; number of steps \\ to update the model: 1; entropy regularization coefficient learned automatically: yes.}\\
\hline
TRPO & \makecell{Policy model: multilayer perceptron; learning rate: 0.001; number of steps to update: 2048; batch size: 128; discount factor: 0.99; \\ maximum number of steps in the Conjugate Gradient algorithm for computing the Hessian vector product: 15; damping in the \\ Hessian vector product computation: 0.1; step-size reduction factor for the line-search: 0.8; maximum number of iteration \\ for the backtracking line-search: 10; number of critic updates per policy update: 10; factor for trade-off of bias vs. variance for \\ Generalized Advantage Estimator: 0.95; target Kullback-Leibler divergence between updates: 0.01.}\\
\hline
\end{tabular}
\end{table*}

The performance of each algorithm for a single run is reported in Table \ref{tab:RLresults}, where
\begin{equation}
\Delta = r_\text{e} - r_\text{s},
\end{equation}
with $r_\text{e}$ and $r_\text{s}$ denoting the mean reward of an episode at the end and at the start of the training, respectively. Each algorithm was tested only once with the exception of DDPG because the first run did not learn within the 200,000 time steps. Overall, the ARS showed a training time significantly smaller than the other methods (1 min 49 sec) and also a final reward $r_\text{e}$ close to the highest (143.0 vs. 143.3, respectively), thus, it is the best performing algorithm according to these single runs followed by A2C. The algorithm with the worst performance is CrossQ as it took 67 minutes and 58 seconds to reach a final reward of just 134.6. 

Note that multiple runs of each RL algorithm should be executed to capture their performance variations, which are a result of the stochasticity of the learning process \cite{NEURIPS2018_7634ea65}. Note also that the training time for an algorithm can differ from one run to another depending on the computational load experienced by the Colab server as visible from the three runs of DDPG. However, the training time variation should be of the order of 5 minutes for a run of 30 minutes and decrease with shorter executions. For example, the execution of ARS (1 min 49 sec) should have a variation of the order of 20 seconds, which confirms that its speed is higher than the other algorithms in this case.      
\begin{table}
\centering
\caption{Training time, $r_\text{s}$, $r_\text{e}$, and $\Delta$ for each RL algorithm. Best values are in bold.}
\label{tab:RLresults}
\begin{tabular}{ccccc} 
 & Tr. time (min:sec) & $r_\text{s}$ & $r_\text{e}$ & $\Delta$ \\ 
\hline
A2C & 9:24 & 112.1 & 142.3 & 30.2 \\
ARS & \textbf{1:49} & 116.2 & 143.0 & 26.8 \\
CrossQ & 67:58 & 112.9 & 134.6 & 21.7\\
\hline
\multirow{3}{*}{DDPG} & 30:42 & 98.0 & 95.4 & -2.6\\
	& 27:44 & 140.1 & \textbf{143.3} & 3.2 \\
	& 26:43 & 136.6 & 142.5 & 5.9\\
\hline
PPO & 7:55 & 116.7 & 134.6 & 17.9\\
SAC & 47:25 & 116.9 & 137.9 & 21.0\\
TQC & 50:19 & 117.0 & 138.0 & 21.0 \\
TRPO & 5:54 & 117.4 & 136.3 & 18.9 \\
\hline
\end{tabular}
\end{table}

Now, let us recall the question we asked at the beginning of this section, i.e., whether it is possible to increase the microalgae uptake of carbon via a proper control strategy. The results in Table \ref{tab:RLresults} show that, after a training phase, the RL controllers learned to generate a light intensity that increases the reward, i.e., the carbon uptake (\ref{eq:carbonUptake}). While this performance of RL is very promising for the design of intelligent autonomous algal systems for carbon control, further investigation is needed to properly understand the behavior of the cultivation when interacting with the RL agent.

\section{Conclusion}\label{sec:concl}
In this paper, we approached net zero as a network control problem, where the source and the sink of carbon dioxide are independently regulated. Since the source, i.e., the anaerobic digester, is an affine-in-control system, we developed an initial-condition-dependent stabilizing controller whose tuning simply requires the desired settling time, but it also requires the existence of $\bm{G}^{-1}(\tilde{\bm{x}})$.

For the sink, i.e., the microalgae system, which is nonaffine-in-control, we designed and compared eight reinforcement-learning (RL) controllers trained to maximize the uptake of carbon dioxide through the light intensity. The controller trained via ARS showed the highest training speed and a competitive final reward, and hence, it was the best algorithm according to the performed single runs. Multiple runs of each RL algorithm should be executed to capture their performance variations, which are a result of the stochasticity of the learning process.     

Overall, while the RL approach to control showed to be more versatile than the nonlinear classical controller, providing finite-time stabilization guarantees with RL is a topic less mature than that predicated on Lyapunov-based nonlinear control. This is related to the wider challenge of explainable RL \cite{milani2024explainable}, which currently makes RL less ready for deployment in real systems compared to Lyapunov-based nonlinear controllers, especially in critical applications. 

With respect to net zero, this work demonstrated the important role that nonlinear and learning-based control can play to reach net zero, especially for improving the efficiency of natural carbon sinks such as microalgae. 

Future work could be providing finite-time stabilization guarantees with RL algorithms, e.g., with ARS, and investigate the behavior of the algal cultivation when controlled by an RL agent towards the creation of intelligent autonomous algal systems for carbon control. Developing and maintaining software libraries for TMN modeling and control is another future research direction.

\ifCLASSOPTIONcaptionsoff
  \newpage
\fi

\bibliographystyle{IEEEtran}
\bibliography{References}

\begin{thebibliography}{10}
\providecommand{\url}[1]{#1}
\csname url@samestyle\endcsname
\providecommand{\newblock}{\relax}
\providecommand{\bibinfo}[2]{#2}
\providecommand{\BIBentrySTDinterwordspacing}{\spaceskip=0pt\relax}
\providecommand{\BIBentryALTinterwordstretchfactor}{4}
\providecommand{\BIBentryALTinterwordspacing}{\spaceskip=\fontdimen2\font plus
\BIBentryALTinterwordstretchfactor\fontdimen3\font minus
  \fontdimen4\font\relax}
\providecommand{\BIBforeignlanguage}[2]{{%
\expandafter\ifx\csname l@#1\endcsname\relax
\typeout{** WARNING: IEEEtran.bst: No hyphenation pattern has been}%
\typeout{** loaded for the language `#1'. Using the pattern for}%
\typeout{** the default language instead.}%
\else
\language=\csname l@#1\endcsname
\fi
#2}}
\providecommand{\BIBdecl}{\relax}
\BIBdecl

\bibitem{EEA}
European\hspace{3pt}Environment\hspace{3pt}Agency, ``Economic losses and
  fatalities from weather- and climate-related events in {Europe},'' last
  access: 4 August 2023. {Available} at:
  \url{https://www.eea.europa.eu/publications/economic-losses-and-fatalities-from}.

\bibitem{ECDCreport}
European\hspace{3pt}Centre\hspace{3pt}for\hspace{3pt}Disease\hspace{3pt}Prevention\hspace{3pt}and\hspace{3pt}Control,
  ``Seasonal influenza -- {Basic} facts,'' last access: 4 August 2023.
  {Available} at:
  \url{https://www.ecdc.europa.eu/sites/default/files/media/en/press/Press\%20Releases/071210_PR_Citizen_factsheet_flu.pdf}.

\bibitem{NASAanomaly}
National\hspace{3pt}Aeronautics\hspace{3pt}and\hspace{3pt}Space\hspace{3pt}Administration\hspace{3pt}(NASA),
  ``Vital signs -- {Global} temperature,'' last access: 4 August 2023.
  {Available} at:
  \url{https://climate.nasa.gov/vital-signs/global-temperature/}.

\bibitem{jong2023increases}
B.-T. Jong, T.~L. Delworth, W.~F. Cooke, K.-C. Tseng, and H.~Murakami,
  ``Increases in extreme precipitation over the {Northeast United States} using
  high-resolution climate model simulations,'' \emph{npj Climate and
  Atmospheric Science}, vol.~6, no.~1, p.~18, 2023.

\bibitem{kendon2023variability}
E.~J. Kendon, E.~M. Fischer, and C.~J. Short, ``Variability conceals emerging
  trend in 100yr projections of {UK} local hourly rainfall extremes,''
  \emph{Nature Communications}, vol.~14, no.~1, p. 1133, 2023.

\bibitem{perkins2020increasing}
S.~Perkins-Kirkpatrick and S.~Lewis, ``Increasing trends in regional
  heatwaves,'' \emph{Nature Communications}, vol.~11, no.~1, p. 3357, 2020.

\bibitem{molina2020future}
M.~Molina, E.~S{\'a}nchez, and C.~Guti{\'e}rrez, ``Future heat waves over the
  {Mediterranean from an Euro-CORDEX} regional climate model ensemble,''
  \emph{Scientific Reports}, vol.~10, no.~1, p. 8801, 2020.

\bibitem{aengenheyster2018point}
M.~Aengenheyster, Q.~Y. Feng, F.~Van Der~Ploeg, and H.~A. Dijkstra, ``The point
  of no return for climate action: {Effects} of climate uncertainty and risk
  tolerance,'' \emph{Earth System Dynamics}, vol.~9, no.~3, pp. 1085--1095,
  2018.

\bibitem{NZdefinition}
University\hspace{3pt}of\hspace{3pt}Oxford, ``What is net zero?'' , last
  access: 1 January 2025. {Available} at:
  \url{https://netzeroclimate.org/what-is-net-zero-2/}.

\bibitem{EU-netZero}
European\hspace{3pt}Commission, ``2050 long-term strategy,'' , last access: 1
  January 2025. {Available} at:
  \url{https://climate.ec.europa.eu/eu-action/climate-strategies-targets/2050-long-term-strategy_en}.

\bibitem{UN-CE}
United\hspace{3pt}Nations, ``Circular economy,'' , last access: 1 January 2025.
  {Available} at:
  \url{https://www.unepfi.org/pollution-and-circular-economy/circular-economy/}.

\bibitem{EU-CE}
European\hspace{3pt}Commission, ``Circular economy action plan,'' , last
  access: 1 January 2025. {Available} at:
  \url{https://environment.ec.europa.eu/strategy/circular-economy-action-plan_en}.

\bibitem{US-CE}
United\hspace{3pt}States\hspace{3pt}Environmental\hspace{3pt}Protection\hspace{3pt}Agency,
  ``What is a circular economy?'' , last access: 1 January 2025. {Available}
  at: \url{https://www.epa.gov/circulareconomy/what-circular-economy}.

\bibitem{potting2017circular}
J.~Potting, M.~P. Hekkert, E.~Worrell, and A.~Hanemaaijer, ``Circular economy:
  {Measuring innovation in the product chain},'' \emph{Planbureau voor de
  Leefomgeving}, no. 2544, 2017, policy report.

\bibitem{GAUSTAD201824}
\BIBentryALTinterwordspacing
G.~Gaustad, M.~Krystofik, M.~Bustamante, and K.~Badami, ``Circular economy
  strategies for mitigating critical material supply issues,'' \emph{Resources,
  Conservation and Recycling}, vol. 135, pp. 24--33, 2018. [Online]. Available:
  \url{https://www.sciencedirect.com/science/article/pii/S0921344917302410}
\BIBentrySTDinterwordspacing

\bibitem{CRMs-EU}
European\hspace{3pt}Commission, ``Critical raw materials,'' , last access: 1
  January 2025. {Available} at:
  \url{https://single-market-economy.ec.europa.eu/sectors/raw-materials/areas-specific-interest/critical-raw-materials_en}.

\bibitem{haddad2019dynamical}
W.~M. Haddad, \emph{{A Dynamical Systems Theory of Thermodynamics}}.\hskip 1em
  plus 0.5em minus 0.4em\relax Princeton, NJ: Princeton University Press, 2019.

\bibitem{haddad2017thermodynamics}
------, ``Thermodynamics: The unique universal science,'' \emph{Entropy},
  vol.~19, no.~11, p. 621, 2017.

\bibitem{zocco2023thermodynamical}
F.~Zocco, P.~Sopasakis, B.~Smyth, and W.~M. Haddad, ``Thermodynamical material
  networks for modeling, planning, and control of circular material flows,''
  \emph{International Journal of Sustainable Engineering}, vol.~16, no.~1, pp.
  1--14, 2023.

\bibitem{zocco2024unification}
F.~Zocco, W.~M. Haddad, A.~Corti, and M.~Malvezzi, ``A unification between
  deep-learning vision, compartmental dynamical thermodynamics, and robotic
  manipulation for a circular economy,'' \emph{IEEE Access}, vol.~12, pp.
  173\,502--173\,516, 2024.

\bibitem{haddad2015finite}
W.~M. Haddad and A.~L'Afflitto, ``Finite-time stabilization and optimal
  feedback control,'' \emph{IEEE Transactions on Automatic Control}, vol.~61,
  no.~4, pp. 1069--1074, 2015.

\bibitem{benemann1997co2}
J.~R. Benemann, ``{$\text{CO}_2$} mitigation with microalgae systems,''
  \emph{Energy Conversion and Management}, vol.~38, pp. S475--S479, 1997.

\bibitem{shafique2020overview}
M.~Shafique, X.~Xue, and X.~Luo, ``An overview of carbon sequestration of green
  roofs in urban areas,'' \emph{Urban Forestry \& Urban Greening}, vol.~47, p.
  126515, 2020.

\bibitem{zhang2022urban}
Y.~Zhang, W.~Meng, H.~Yun, W.~Xu, B.~Hu, M.~He, X.~Mo, and L.~Zhang, ``Is urban
  green space a carbon sink or source? {- A case study of China based on LCA}
  method,'' \emph{Environmental Impact Assessment Review}, vol.~94, p. 106766,
  2022.

\bibitem{marchi2015carbon}
M.~Marchi, R.~M. Pulselli, N.~Marchettini, F.~M. Pulselli, and S.~Bastianoni,
  ``Carbon dioxide sequestration model of a vertical greenery system,''
  \emph{Ecological Modelling}, vol. 306, pp. 46--56, 2015.

\bibitem{wang2021promoting}
Y.~Wang, Q.~Chang, and X.~Li, ``Promoting sustainable carbon sequestration of
  plants in urban greenspace by planting design: A case study in parks of
  {Beijing},'' \emph{Urban Forestry \& Urban Greening}, vol.~64, p. 127291,
  2021.

\bibitem{zhou2017bio}
W.~Zhou, J.~Wang, P.~Chen, C.~Ji, Q.~Kang, B.~Lu, K.~Li, J.~Liu, and R.~Ruan,
  ``Bio-mitigation of carbon dioxide using microalgal systems: Advances and
  perspectives,'' \emph{Renewable and Sustainable Energy Reviews}, vol.~76, pp.
  1163--1175, 2017.

\bibitem{ramaraj2015biomass}
R.~Ramaraj, D.~D.-W. Tsai, and P.~H. Chen, ``Biomass of algae growth on natural
  water medium,'' \emph{Journal of Photochemistry and Photobiology B: Biology},
  vol. 142, pp. 124--128, 2015.

\bibitem{viswanaathan2022integrated}
S.~Viswanaathan, P.~K. Perumal, and S.~Sundaram, ``Integrated approach for
  carbon sequestration and wastewater treatment using algal--bacterial
  consortia: Opportunities and challenges,'' \emph{Sustainability}, vol.~14,
  no.~3, p. 1075, 2022.

\bibitem{arun2021technical}
J.~Arun, K.~P. Gopinath, R.~Sivaramakrishnan, P.~SundarRajan, R.~Malolan, and
  A.~Pugazhendhi, ``Technical insights into the production of green fuel from
  {$\text{CO}_2$} sequestered algal biomass: A conceptual review on green
  energy,'' \emph{Science of the Total Environment}, vol. 755, p. 142636, 2021.

\bibitem{suarez2019operational}
B.~Su{\'a}rez-Eiroa, E.~Fern{\'a}ndez, G.~M{\'e}ndez-Mart{\'\i}nez, and
  D.~Soto-O{\~n}ate, ``Operational principles of circular economy for
  sustainable development: Linking theory and practice,'' \emph{Journal of
  Cleaner Production}, vol. 214, pp. 952--961, 2019.

\bibitem{sarwer2022algal}
A.~Sarwer, S.~M. Hamed, A.~I. Osman, F.~Jamil, A.~H. Al-Muhtaseb, N.~S.
  Alhajeri, and D.~W. Rooney, ``Algal biomass valorization for biofuel
  production and carbon sequestration: A review,'' \emph{Environmental
  Chemistry Letters}, vol.~20, no.~5, pp. 2797--2851, 2022.

\bibitem{EMAFund}
Ellen\hspace{3pt}MacArthur\hspace{3pt}Foundation, ``What is a circular
  economy?'' , last access: 1 January 2025. {Available} at:
  \url{https://www.ellenmacarthurfoundation.org/topics/circular-economy-introduction/overview}.

\bibitem{amicarelli2022life}
V.~Amicarelli, C.~Bux, M.~P. Spinelli, and G.~Lagioia, ``Life cycle assessment
  to tackle the take-make-waste paradigm in the textiles production,''
  \emph{Waste Management}, vol. 151, pp. 10--27, 2022.

\bibitem{walker2020life}
S.~Walker and R.~Rothman, ``Life cycle assessment of bio-based and fossil-based
  plastic: A review,'' \emph{Journal of Cleaner Production}, vol. 261, p.
  121158, 2020.

\bibitem{xia2022review}
X.~Xia and P.~Li, ``A review of the life cycle assessment of electric vehicles:
  Considering the influence of batteries,'' \emph{Science of the Total
  Environment}, vol. 814, p. 152870, 2022.

\bibitem{cucurachi2019life}
S.~Cucurachi, L.~Scherer, J.~Guin{\'e}e, and A.~Tukker, ``Life cycle assessment
  of food systems,'' \emph{One Earth}, vol.~1, no.~3, pp. 292--297, 2019.

\bibitem{sala2021evolution}
S.~Sala, A.~M. Amadei, A.~Beylot, and F.~Ardente, ``The evolution of life cycle
  assessment in european policies over three decades,'' \emph{The International
  Journal of Life Cycle Assessment}, vol.~26, pp. 2295--2314, 2021.

\bibitem{luan2021dynamic}
X.~Luan, X.~Cui, L.~Zhang, X.~Chen, X.~Li, X.~Feng, L.~Chen, W.~Liu, and
  Z.~Cui, ``Dynamic material flow analysis of plastics in {China} from 1950 to
  2050,'' \emph{Journal of Cleaner Production}, vol. 327, p. 129492, 2021.

\bibitem{li2022uncovering}
M.~Li, Y.~Geng, G.~Liu, Z.~Gao, X.~Rui, and S.~Xiao, ``Uncovering
  spatiotemporal evolution of titanium in {China}: {A dynamic material flow
  analysis},'' \emph{Resources, Conservation and Recycling}, vol. 180, p.
  106166, 2022.

\bibitem{sieber2020dynamic}
R.~Sieber, D.~Kawecki, and B.~Nowack, ``Dynamic probabilistic material flow
  analysis of rubber release from tires into the environment,''
  \emph{Environmental Pollution}, vol. 258, p. 113573, 2020.

\bibitem{liu2021dynamic}
W.~Liu, W.~Liu, X.~Li, Y.~Liu, A.~E. Ogunmoroti, M.~Li, M.~Bi, and Z.~Cui,
  ``Dynamic material flow analysis of critical metals for lithium-ion battery
  system in {China} from 2000--2018,'' \emph{Resources, Conservation and
  Recycling}, vol. 164, p. 105122, 2021.

\bibitem{eriksen2020dynamic}
M.~K. Eriksen, K.~Pivnenko, G.~Faraca, A.~Boldrin, and T.~F. Astrup, ``Dynamic
  material flow analysis of {PET, PE, and PP flows in Europe: Evaluation} of
  the potential for circular economy,'' \emph{Environmental Science \&
  Technology}, vol.~54, no.~24, pp. 16\,166--16\,175, 2020.

\bibitem{zocco2022circularity}
F.~Zocco, ``Circularity of thermodynamical material networks: Indicators,
  examples, and algorithms,'' \emph{arXiv preprint arXiv:2209.15051}, 2024.

\bibitem{zocco2024circular}
F.~Zocco and M.~Malvezzi, ``Circular economy design through system dynamics
  modeling,'' \emph{arXiv preprint arXiv:2411.13540}, 2024.

\bibitem{Roxin1966}
E.~Roxin, ``On finite stability in control systems,'' \emph{Rendiconti del
  Circolo Matematico di Palermo}, vol.~15, no.~3, pp. 273--282, 1966.

\bibitem{moulay2008finite}
E.~Moulay and W.~Perruquetti, ``Finite time stability conditions for
  non-autonomous continuous systems,'' \emph{International Journal of Control},
  vol.~81, no.~5, pp. 797--803, 2008.

\bibitem{haddad2008finite}
W.~M. Haddad, S.~G. Nersesov, and L.~Du, ``Finite-time stability for
  time-varying nonlinear dynamical systems,'' in \emph{2008 American control
  conference}.\hskip 1em plus 0.5em minus 0.4em\relax IEEE, 2008, pp.
  4135--4139.

\bibitem{bhat1998continuous}
S.~P. Bhat and D.~S. Bernstein, ``Continuous finite-time stabilization of the
  translational and rotational double integrators,'' \emph{IEEE Transactions on
  Automatic Control}, vol.~43, no.~5, pp. 678--682, 1998.

\bibitem{hong2001output}
Y.~Hong, J.~Huang, and Y.~Xu, ``On an output feedback finite-time stabilization
  problem,'' \emph{IEEE Transactions on Automatic Control}, vol.~46, no.~2, pp.
  305--309, 2001.

\bibitem{haddad2023finite}
W.~M. Haddad and J.~Lee, ``Finite-time stabilization and optimal feedback
  control for nonlinear discrete-time systems,'' \emph{IEEE Transactions on
  Automatic Control}, vol.~68, no.~3, pp. 1685--1691, 2023.

\bibitem{lee2023finite}
J.~Lee, W.~M. Haddad, and M.~Lanchares, ``Finite time stability and optimal
  finite time stabilization for discrete-time stochastic dynamical systems,''
  \emph{IEEE Transactions on Automatic Control}, vol.~68, no.~7, pp.
  3978--3991, 2023.

\bibitem{pmlr-v37-schulman15}
J.~Schulman, S.~Levine, P.~Abbeel, M.~Jordan, and P.~Moritz, ``Trust region
  policy optimization,'' in \emph{Proceedings of the 32nd International
  Conference on Machine Learning}, 2015, pp. 1889--1897.

\bibitem{mnih2016asynchronous}
V.~Mnih, A.~P. Badia, M.~Mirza, A.~Graves, T.~Harley, T.~P. Lillicrap,
  D.~Silver, and K.~Kavukcuoglu, ``Asynchronous methods for deep reinforcement
  learning,'' in \emph{Proceedings of the 33rd International Conference on
  Machine Learning-Volume 48}, 2016, pp. 1928--1937.

\bibitem{lillicrap2019continuouscontroldeepreinforcement}
\BIBentryALTinterwordspacing
T.~P. Lillicrap, J.~J. Hunt, A.~Pritzel, N.~Heess, T.~Erez, Y.~Tassa,
  D.~Silver, and D.~Wierstra, ``Continuous control with deep reinforcement
  learning,'' 2015. [Online]. Available: \url{https://arxiv.org/abs/1509.02971}
\BIBentrySTDinterwordspacing

\bibitem{schulman2017proximal}
J.~Schulman, F.~Wolski, P.~Dhariwal, A.~Radford, and O.~Klimov, ``Proximal
  policy optimization algorithms,'' \emph{arXiv preprint arXiv:1707.06347},
  2017.

\bibitem{NEURIPS2018_7634ea65}
\BIBentryALTinterwordspacing
H.~Mania, A.~Guy, and B.~Recht, ``Simple random search of static linear
  policies is competitive for reinforcement learning,'' in \emph{Advances in
  Neural Information Processing Systems (NeurIPS 2018)}, vol.~31, 2018, pp.
  1--10. [Online]. Available:
  \url{https://proceedings.neurips.cc/paper_files/paper/2018/file/7634ea65a4e6d9041cfd3f7de18e334a-Paper.pdf}
\BIBentrySTDinterwordspacing

\bibitem{bhattcrossq}
A.~Bhatt, D.~Palenicek, B.~Belousov, M.~Argus, A.~Amiranashvili, T.~Brox, and
  J.~Peters, ``{CrossQ: Batch} normalization in deep reinforcement learning for
  greater sample efficiency and simplicity,'' in \emph{The Twelfth
  International Conference on Learning Representations (ICLR 2024)}, pp. 1--19.

\bibitem{kokolakis2023fixed}
N.-M.~T. Kokolakis, K.~G. Vamvoudakis, and W.~M. Haddad, ``Fixed-time learning
  for optimal feedback control,'' in \emph{International Design Engineering
  Technical Conferences and Computers and Information in Engineering
  Conference}, vol. 87387.\hskip 1em plus 0.5em minus 0.4em\relax American
  Society of Mechanical Engineers, 2023.

\bibitem{abel2024definition}
D.~Abel, A.~Barreto, B.~Van~Roy, D.~Precup, H.~P. van Hasselt, and S.~Singh,
  ``A definition of continual reinforcement learning,'' \emph{Advances in
  Neural Information Processing Systems}, vol.~36, 2024.

\bibitem{bernard2001dynamical}
O.~Bernard, Z.~Hadj-Sadok, D.~Dochain, A.~Genovesi, and J.-P. Steyer,
  ``Dynamical model development and parameter identification for an anaerobic
  wastewater treatment process,'' \emph{Biotechnology and Bioengineering},
  vol.~75, no.~4, pp. 424--438, 2001.

\bibitem{campos2019hybrid}
A.~Campos-Rodr{\'\i}guez, J.~Garc{\'\i}a-Sandoval, V.~Gonz{\'a}lez-{\'A}lvarez,
  and A.~Gonz{\'a}lez-{\'A}lvarez, ``Hybrid cascade control for a class of
  nonlinear dynamical systems,'' \emph{Journal of Process Control}, vol.~76,
  pp. 141--154, 2019.

\bibitem{UKlawDigesters}
\BIBentryALTinterwordspacing
UK\hspace{3pt}Environment\hspace{3pt}Agency, ``Treating, storing, and using
  carbon dioxide from anaerobic digestion: {RPS} 255,'' last access: 13 July
  2023\\. [Online]. Available:
  \url{https://www.gov.uk/government/publications/treating-storing-and-using-carbon-dioxide-from-anaerobic-digestion-rps-255/treating-storing-and-using-carbon-dioxide-from-anaerobic-digestion-rps-255}
\BIBentrySTDinterwordspacing

\bibitem{vatcheva2006experiment}
I.~Vatcheva, H.~De~Jong, O.~Bernard, and N.~J. Mars, ``Experiment selection for
  the discrimination of semi-quantitative models of dynamical systems,''
  \emph{Artificial Intelligence}, vol. 170, no. 4-5, pp. 472--506, 2006.

\bibitem{marcos2004output}
N.~Marcos, M.~Guay, and D.~Dochain, ``Output feedback adaptive extremum seeking
  control of a continuous stirred tank bioreactor with {Monod's kinetics},''
  \emph{Journal of Process Control}, vol.~14, no.~7, pp. 807--818, 2004.

\bibitem{solimeno2015new}
A.~Solimeno, R.~Sams{\'o}, E.~Uggetti, B.~Sialve, J.-P. Steyer, A.~Gabarr{\'o},
  and J.~Garc{\'\i}a, ``New mechanistic model to simulate microalgae growth,''
  \emph{Algal Research}, vol.~12, pp. 350--358, 2015.

\bibitem{bernard2011hurdles}
O.~Bernard, ``Hurdles and challenges for modelling and control of microalgae
  for {CO2} mitigation and biofuel production,'' \emph{Journal of Process
  Control}, vol.~21, no.~10, pp. 1378--1389, 2011.

\bibitem{hall1999photosynthesis}
D.~O. Hall and K.~Rao, \emph{Photosynthesis}.\hskip 1em plus 0.5em minus
  0.4em\relax Cambridge, U.K.: Cambridge University Press, 1999.

\bibitem{darko2014photosynthesis}
E.~Darko, P.~Heydarizadeh, B.~Schoefs, and M.~R. Sabzalian, ``Photosynthesis
  under artificial light: {The} shift in primary and secondary metabolism,''
  \emph{Philosophical Transactions of the Royal Society B: Biological
  Sciences}, vol. 369, no. 1640, p. 20130243, 2014.

\bibitem{stable-baselines3}
\BIBentryALTinterwordspacing
A.~Raffin, A.~Hill, A.~Gleave, A.~Kanervisto, M.~Ernestus, and N.~Dormann,
  ``{Stable-Baselines3}: {Reliable} reinforcement learning implementations,''
  \emph{Journal of Machine Learning Research}, vol.~22, no. 268, pp. 1--8,
  2021. [Online]. Available: \url{http://jmlr.org/papers/v22/20-1364.html}
\BIBentrySTDinterwordspacing

\bibitem{haarnoja2018soft}
T.~Haarnoja, A.~Zhou, P.~Abbeel, and S.~Levine, ``{Soft actor-critic:
  Off-policy} maximum entropy deep reinforcement learning with a stochastic
  actor,'' in \emph{International Conference on Machine Learning}.\hskip 1em
  plus 0.5em minus 0.4em\relax PMLR, 2018, pp. 1861--1870.

\bibitem{kuznetsov2020controlling}
A.~Kuznetsov, P.~Shvechikov, A.~Grishin, and D.~Vetrov, ``Controlling
  overestimation bias with truncated mixture of continuous distributional
  quantile critics,'' in \emph{International Conference on Machine
  Learning}.\hskip 1em plus 0.5em minus 0.4em\relax PMLR, 2020, pp. 5556--5566.

\bibitem{milani2024explainable}
S.~Milani, N.~Topin, M.~Veloso, and F.~Fang, ``Explainable reinforcement
  learning: {A} survey and comparative review,'' \emph{ACM Computing Surveys},
  vol.~56, no.~7, pp. 1--36, 2024.

\end{thebibliography}

\end{document}